\documentclass[12pt]{amsart}
\usepackage{amsmath}
\usepackage{amssymb}
\usepackage{amsthm}
\usepackage[utf8]{inputenc}
\usepackage[english]{babel}
\usepackage{color}
\usepackage{graphicx}

\newcommand \N{\mathbb{N}}
\newcommand \X{X}

\newcommand \R{\mathbb{R}}

\newcommand \B{\mathcal{B}}

\newtheorem{teo}{Theorem}[section]
\newtheorem{lema}[teo]{Lemma}
\newtheorem{cor}[teo]{Corollary}
\newtheorem{pro}[teo]{Proposition}
\newtheorem{defi}[teo]{Definition}
\newtheorem{ob}[teo]{Remark}

\newtheorem{que}[teo]{Question}
\newtheorem{nota}[teo]{Notation}

	\title[Distribution function on a linearly ordered space]{The distribution function of a probability measure on a linearly ordered topological space}

\author{J.F. G\'alvez-Rodr\'{\i}guez \and M.A. S\'anchez-Granero}

\email{jgr409@ual.es and misanche@ual.es}

\thanks{The second author is supported by grant MTM2015-64373-P (MINECO/FEDER, UE)}

\curraddr{Department of Mathematics\\
	University of Almer\'{\i}a\\ 04120 Almer\'{\i}a (Spain)}

\subjclass[2010]{Primary 60E05; Secondary 60B11}
\keywords{probability; measure; $\sigma$-algebra; Borel $\sigma$-algebra; distribution function; cumulative distribution function; sample; linearly ordered topological space}

\begin{document}

	\begin{abstract}
		In this paper we describe a theory of a cumulative distribution function on a  space with an order from a probability measure defined in this space. This distribution function plays a similar role to that played in the classical case. Moreover, we define its pseudo-inverse and study its properties. Those properties will allow us to generate samples of a distribution and give us the chance to calculate integrals with respect to the related probability measure. 
	\end{abstract}

\maketitle

\section{Introduction}

This work collects some results on a theory of a cumulative distribution function on a linearly ordered topological space (LOTS).

Moreover, we show that this function plays a similar role to that played in the classical case and study its inverse, which allows us to generate samples of the probability measure that we use to define the distribution function.

The main goal of this paper is to  provide a theory of a cumulative distribution function on a space with a linear order. What is more, we show that a cumulative distribution function in this context plays a similar role to that played by a distribution function in the classical case. Recall that, in the classical case, the cumulative distribution function (in short, cdf) of a real-valued random variable $X$ is the function given by $F_X(x)=P[X \leq x]$ and it satisfies the following properties:

\begin{enumerate}
	\item $F$ is non-decreasing, what means that for each $x, y \in \R$ with $x<y$, we have $F(x) \leq F(y)$.
	\item $F$ is right-continuous, what means that $F(a)=\lim_{x \rightarrow a^+} F(x)$, for each $a \in \R$. Furthermore, it follows that
	$\lim_{x \rightarrow -\infty} F(x)=0$ and $\lim_{x \rightarrow +\infty} F(x)=1$.
\end{enumerate}

Moreover, given a cdf in an ordered space, we define its pseudo-inverse and study its properties. In the classical case, if the cdf $F$ is strictly increasing and continuous then $ F^{-1}(p),p\in [0,1],$ is the unique real number $x$ such that $F(x)=p$. In such a case, this defines the inverse distribution function.

Some distributions do not have a unique inverse (for example in the case where the density function ${\displaystyle f_{X}(x)=0}$ for all ${\displaystyle a<x<b}$, causing ${\displaystyle F_{X}}$ to be constant). This problem can be solved by defining, for ${\displaystyle p\in [0,1]}$, the pseudo-inverse distribution function:

$ F^{-1}(p)=\inf\{x\in \mathbb {R} :F(x)\geq p\}$

The inverse of a cdf let us generate samples of a distribution. Indeed, let $X$ be a random variable whose distribution can be described by the cumulative distribution function $F$. We want to generate values of $X$ which are distributed according to this distribution. The inverse transform sampling method works as follows: generate a random number $u$ from the standard uniform distribution in the interval $[0,1]$ and then take $x=F^{-1}(u)$.

Roughly speaking, given a continuous uniform variable $U$ in $[0,1]$ and a cumulative distribution function $F$, the random variable $X=F^{-1}(U)$ has distribution $F$ (or, $X$ is distributed $F$).

For further reference about the pseudo-inverse of $F$ see, for example, \cite[Chapter 1]{Baz}.

In our context, we study the pseudo-inverse of a cdf. This pseudo-inverse allows us to generate samples of the distribution as well as to calculate integrals with respect to the related probability measure.

\section{Preliminaries}

\subsection{Measure theory}

Now we recall some definitions related to measure theory from \cite{Halmos}.
Let $X$ be a set, then there are several classes of sets of $X$. If $\mathcal{R}$ is a non-empty collection of subsets of $X$, we say that $\mathcal{R}$ is a ring if it is closed under complement and finite union. What is more, given $\mathcal{Q}$ is a non-empty collection of subsets of $X$ it is said to be an algebra if it is a ring such that $X \in \mathcal{Q}$. Moreover, a non-empty collection of subsets of $X$, $\mathcal{A}$, is a $\sigma$-algebra if it is closed under complement and countable union and $X \in \mathcal{A}$.

For a given topological space, $(X, \tau)$,  $\mathcal{B}=\sigma(\tau)$ is the Borel $\sigma$-algebra of the space, that is, it is the $\sigma$-algebra generated by the open sets of $X$.

\begin{defi}
	Given a measurable space $(\Omega, \mathcal{A})$, a measure $\mu$ is a non-negative and $\sigma$-additive set mapping defined on $\mathcal{A}$ such that $\mu(\emptyset)=0$. 
\end{defi}

A set mapping is said to be $\sigma$-additive if $\mu(\bigcup_{n=1}^\infty A_n)=\sum_{n=1}^\infty \mu(A_n)$ for each countable collection $\{A_n\}_{n=1}^\infty$ of pairwise disjoint sets in $\mathcal{A}$.

Each measure is monotonous, what means that $\mu(A) \leq \mu(B)$, for each $A \subseteq B$. Moreover it is continuous in the next sense: if $A_n \rightarrow A$, then $\mu(A_n) \rightarrow \mu(A)$. What is more, if $A_n$ is a monotically non-decreasing sequence of sets (what means that $A_n \subseteq A_{n+1}$, for each $n \in \N$) then $\mu(A_n) \rightarrow \mu(\bigcup_{n \in \N} A_n)$. If $A_n$ is monotically non-increasing (what means that $A_{n+1} \subseteq A_{n}$, for each $n \in \N$), then $\mu(A_n) \rightarrow \mu(\bigcap_{n \in \N} A_n)$.

\subsection{Ordered sets}

First we recall the definition of a linear order and a linearly ordered topological space:

\begin{defi}(\cite[Chapter 1]{Roman})
	A partially ordered set $(P, \leq)$ (that is, a set $P$ with the binary relation $\leq$ that is reflexive, antisymmetric and transitive) is totally ordered if every $x, y \in P$ are comparable, that is, $x \leq y$ or $y \leq x$. In this case, the order is said to be total or linear. 
\end{defi}

For further reference about partially ordered sets see, for example, \cite{Dushnik}.

\begin{defi}(\cite[Section 1]{Lutzer})
	A linearly ordered topological space (abbreviated LOTS) is a triple $(X, \tau, \leq)$ where $(X, \leq)$ is a linearly ordered set and where $\tau$ is the topology of the order $\leq$.
\end{defi}

The definition of the order topology is the following one:

\begin{defi}(\cite[Part II, 39]{Counter})
	Let $X$ be a set which is linearly ordered by $<$, we define the order topology $\tau$ on $X$ by taking the subbasis $\{\{x \in X: x<a\}: a \in X\} \cup \{\{x \in X: x>a\}: a \in X\}$.  
\end{defi}

From a linear order, $\leq$, in $X$ we define

\begin{defi}
	Let $a, b \in X$ with $a \leq b$, we define the set $]a,b]=\{x \in \X: a < x \leq b\}$. Analogously, we define $]a,b[, [a,b]$ and $[a,b[$. Moreover, $(\leq a)$ is given by $(\leq a)=\{x \in \X: x \leq a\}$. $(<a), (\geq a)$ and $(>a)$ are defined similarly. 
\end{defi}

\begin{nota}
	Let $a \in X$, we will also use $(a,\infty)$ and $[a,\infty)$ to denote $(>a)$ and $(\geq a)$, respectively. Similarly, $(-\infty,a)$ and $(-\infty,a]$ will also denote $(<a)$ and $\leq a$, respectively.
\end{nota}

\begin{ob}\label{ob:base}
	Note that an open basis of $X$ with respect to $\tau$ is given by $\{]a,b[: a<b, a, b \in (X \cup \{-\infty,\infty\}) \}$.
\end{ob}

For our study we need to introduce some terminology.

\begin{defi}(\cite[Section 1]{Lutzer})
	Let $(X, \leq)$ be a linearly ordered set. A subset $C \subseteq X$ is said to be convex in $X$ if, whenever $a, b \in C$ with $a \leq b$, then $\{x \in X: a \leq x \leq b\}$ is a subset of $C$. 
\end{defi}

\begin{pro}(\cite[Part II, 39]{Counter})\label{pro:convexcomp}
	Any subset $A \subseteq X$ can be uniquely expressed as an union of disjoint, nonempty, maximal convex sets in $A$, called convex components.
\end{pro}

 The definition of interval is the following one.

\begin{defi}(\cite[Section 1]{Lutzer})
	An interval of $X$ is a convex subset of $X$ having two endpoints in $X$, which can belong to the interval or not. 
\end{defi}

For convention, we will assume that $\infty$ and $-\infty$ can be the endpoints of intervals.

\begin{defi}(\cite[Defs 2.16, 2.17]{order})
	Let $P$ be an ordered set and let $A \subseteq P$. Then:
	\begin{enumerate}
		\item $l$ is called a lower bound of $A$ if, and only if we have $l \leq a$, for each $a \in A$.
		\item $u$ is called an upper bound of $A$ if, and only if we have $u \geq a$, for each $a \in A$.
	\end{enumerate}
\end{defi}

\begin{defi}
	Given $A \subseteq X$, we denote by $A^l$ and $A^u$, respectively, the set of lower and upper bounds of $A$.
\end{defi}

\begin{defi}(\cite[Def 3.18]{order})
	Let $P$ be an ordered set and let $A \subseteq P$. Then:
	\begin{enumerate}
		\item The point $u$ is called the lowest upper bound or supremum or join of $A$ iff $u$ is the minimum of the set $A^u$.
			\item The point $u$ is called the greatest lower bound or infimum or meet of $A$ iff $l$ is the maximum of the set $A^l$.
	\end{enumerate}
\end{defi}

\begin{pro}(\cite[Part II, 39]{Counter})\label{pro:complete}
	The order topology on $X$ is compact if, and only if the order is complete, that is, if, and only if, every nonempty subset of $X$ has a greatest lower bound and a least upper bound.
\end{pro}

\begin{ob}
In the rest of the paper, unless otherwise stated, $X$ will be a separable LOTS and a measure in $X$ will be with respect to the Borel $\sigma$-algebra of $X$.
\end{ob}

\section{The order in $X$}

In this section we study some properties (mainly topological) of a separable LOTS.

The definition of the topology $\tau$ suggest the next

\begin{defi}
	Let $x \in \X$, it is said to be a left-isolated (respectively right- isolated) point if $(<x)=\emptyset$ (respectively $(>x)=\emptyset$) or there exists $z \in \X$ such that $]z,x[ =\emptyset$  (respectively there exists $z \in X$ such that $]x,z[=\emptyset$). Moreover, we will say that $x \in \X$ is isolated if it is both right and left-isolated.
\end{defi}

\begin{lema}\label{lema:infA}
	Let $A, B \subseteq X$ be such that $A^l=B^l$ (respectively $A^u=B^u$). If there exists $\inf A$ (respectively $\sup A$), then there exists $\inf B$ (respectively $\sup B$) and $\inf A=\inf B$ (respectively $\sup A=\sup B$). 
	
	\begin{proof}
		
		Let $A, B\subseteq X$ be such that $A^l=B^l$ and suppose that there exists $\inf A$. It holds that $x \leq \inf A$, for each $x \in A^l$. Now, since $A^l=B^l$, we have that $\inf A \in A^l=B^l$ and $x \leq \inf A$, for each $x \in B^l$, that is, $\inf A=\inf B$. 
		
		The case in which $A^u=B^u$ and there exists $\sup A$ can be proven analogously.
		\end{proof}
\end{lema}
\begin{pro}\label{pro:leftboundseq}
	Let $A \subseteq X$ be a nonempty subset such that it does not have a minimum (respectively a maximum), then there exists a sequence $a_n \in A$ such that $a_{n+1}<a_n$, for each $n \in \N$ and $A^l=\{a_n: n \in \N\}^l$ (respectively $a_{n+1}>a_n$, for each $n \in \N$ and $A^u=\{a_n: n \in \N\}^u$).
	\begin{proof}
		Let $D$ be a dense and countable subset of $X$ and consider$D_A=\{d \in D: d \notin A^l\}$. Note that the fact that $d \notin A^l$ is equivalent to the existence of $a \in A$ such that $a<d$. Moreover, $D_A \subseteq D$, so $D_A$ is countable, so we can enumerate it as $D_A=\{d_n: n \in \N\}$. Given $d_1 \in D_A$, there exists $a_1 \in A$ such that $a_1<d_1$. Suppose that $a_n \in A$ is a sequence defined by $a_n <d_n$ and $a_n<a_{n-1}$, for each $n \in \N$. We define $a_{n+1}$ as follows. Since there does not exists the minimum of $A$, we can choose $a \in A$ such that $a<a_n$. Apart from that, there exists $a' \in A$ such that $a'<d_{n+1}$. Hence, if we consider $a_{n+1}=\min\{a, a'\}$, then $a_{n+1}<a_n$ and $a_{n+1}<d_{n+1}$. Recursively, we have defined a sequence $a_n \in A$ such that $a_{n+1}<a_n$ and $a_{n}<d_n$, for each $n \in \N$.
		
		Now we prove that $A^l=\{a_n: n \in \N\}^l$.
		
		$\subseteq)$ This is obvious.
		
		$\supseteq)$ Let $x \in X$ be such that $x \leq a_n$, for each $n \in \N$. Now we prove that $x \leq a$, for each $a \in A$. For that purpose, let $a \in A$. Since there does not exist the minimum of $A$, there exist $a' \in A$ such that $a'<a$ and $a'' \in A$ such that $a''<a'$. Consequently, $]a'',a[$ is a nonempty open set in $X$ with respect to $\tau$, so we can choose $d \in D \cap ]a'',a[$. Hence, $d>a''$, what implies that $d \in D_A$. It follows that there exists $n_0 \in \N$ such that $d=d_{n_0}$. Therefore, $x \leq a_{n_0}<d_{n_0}<a$, what let us conclude that $x \leq a$.
		\end{proof}
\end{pro}

Convex subsets can be described as countable union of intervals.

\begin{cor}\label{cor:convex}
	Let $A \subseteq X$ be a convex subset. Then it holds that:
	\begin{enumerate}
		\item If there exist both minimum and maximum of $A$, then $A=[\min A, \max A]$.
		\item If there does not exist the minimum of $A$ but it does its maximum, then there exists a decreasing sequence $a_n \in A$ such that $A=\bigcup_{n \in \N} ]a_n, \max A]$.
		\item If there does not exist the maximum of $A$ but it does its minimum, then there exists an increasing sequence $b_n \in A$ such that $A=\bigcup_{n \in \N} [\min A, b_n[$.
		\item If there does not exist the minimum of $A$ nor its maximum, then there exist a decreasing sequence $a_n \in A$ and an increasing one $b_n \in A$ such that $A=\bigcup_{n \in \N} ]a_n, b_n[$.
	\end{enumerate}

\begin{proof}
	\begin{enumerate}
		\item It is clear.
		\item Since $A$ is nonempty and there does not exist the minimum of $A$, by Proposition \ref{pro:leftboundseq}, we can choose a sequence $a_n \in A$ such that $a_{n+1}<a_n$, for each $n \in \N$ and $A^l=\{a_n: n \in \N\}^l$. Now we prove that $A=\bigcup_{n \in \N} ]a_n,\max A]$.
		
		$\subseteq)$ Let $x \in A$. Since $A$ does not have a minimum, then $x \notin A^l$ what implies that $x \notin \{a_n: n \in \N\}^l$. Then there exists $n \in \N$ such that $a_n<x$. Consequently, $x \in \bigcup_{n \in \N} ]a_n, \max A]$.
		
		$\supseteq)$ Let $x \in \bigcup_{n \in \N} ]a_n, \max A]$, then there exists $n \in \N$, such that $a_n<x\leq \max A$. Hence, the fact that $A$ is convex together with the fact that $a_n \in A$ give us that $x \in A$.
		
		\item It can be proven similarly to the previous item.
		
		\item Since $A$ is nonempty and there does not exist the minimum of $A$ nor its maximum, by Proposition \ref{pro:leftboundseq}, we can choose two sequences $a_n, b_n \in A$ such that $a_{n+1}<a_n$ and $b_{n+1}>b_n$, for each $n \in \N$ and $A^l=\{a_n: n \in \N\}^l$, $B^u=\{b_n: n \in \N\}^u$. Now we prove that $A=\bigcup_{n \in \N} ]a_n,b_n[$.
		
		$\subseteq)$ Let $x \in A$. Since $A$ does not have a minimum nor a maximum then $x \notin A^l$ and $x \notin A^u$, what implies that $x \notin \{a_n:n \in \N\}^l$ and $x \notin \{b_n: n \in \N\}^u$, then there exists $n_1 \in \N$ and $n_2 \in \N$ such that $a_{n_1}<x<b_{n_2}$. If we define $n=\max\{n_1, n_2\}$, then it holds that $a_{n}<x<b_n$ and we conclude that $x \in \bigcup_{n \in \N} ]a_n,b_n[$.
		
		$\supseteq)$ Let $x \in \bigcup_{n \in \N} ]a_n, b_n[$, then there exists $n \in \N$, such that $a_n<x<b_n$. Hence, the fact that $A$ is convex together with the fact that $a_n, b_n \in A$ give us that $x \in A$.
		
	\end{enumerate}
\end{proof}
\end{cor}

Similarly, convex open subsets can be described as countable union of open intervals.

\begin{cor}\label{cor:intervals}
	Let $A$ be an open and convex subset of $X$, then $A$ is the countable union of open intervals.
	\begin{proof}
		We distinguish some cases depending on whether there exist the maximum or the minimum of $A$:
		
		\begin{enumerate}
			\item Suppose that there does not exist the maximum of $A$ nor its minimum, then by Corollary \ref{cor:convex}.4, it holds that $A$ can be written as the countable union of open intervals.
			\item Suppose that there does not exist the minimum of $A$ but it does its maximum. By the previous corollary, it holds that $A=\bigcup_{n \in \N} ]a_n, \max A]$. Now note that the fact that $A$ is open means that $\max A$ is right-isolated so we can write $A=\bigcup_{n \in \N} ]a_n, b[$, where $b$ is the following point to $\max A$. Hence, $A$ is the countable union of open intervals.
			\item If there exists the minimum of $A$ but not its maximum, we can proceed anaologously to claim that $A=\bigcup_{n \in \N}]a,b_n[$ where $a$ is the previous point to $\min A$ and $b_n$ is an increasing sequence in $A$.
			\item If there exists both minimum and maximum of $A$, then $A=]a,b[$ where $a$ is the previous point to $\min A$ and $b$ is the following one to $\max A$.
		\end{enumerate}
		\end{proof}
\end{cor}

Next, we prove that a separable LOTS is first countable.

\begin{pro}\label{pro:ANI}
	$\tau$ is first countable.
	\begin{proof}
		Since $X$ is separable with respect to the topology $\tau$, there exists a countable dense subset $D$ of $X$. Now we prove that, given $x \in X$, each of the countable families
		\begin{itemize}
			\item[-] $\mathcal{B}_{x}=\{\{x\}\}$ if $x$ is isolated.
			\item[-] $\mathcal{B}_{x}=\{]a,b[: a<x<b, a, b \in D\}$, if $x$ is not left-isolated nor right-isolated.
			\item[-] $\mathcal{B}_{x}=\{[x, b[: x<b, b \in D\}$ if $x$ is left-isolated but it is not right-isolated.
			\item[-] $\mathcal{B}_{x}=\{]a,x]: a<x, a \in D\}$ if $x$ is right-isolated but it is not left-isolated.
		\end{itemize}
		is a countable neighborhood basis of $x$ with respect to the topology $\tau$. For that purpose we prove the next two items:
		\begin{itemize}
			\item Each element of $\mathcal{B}_{x}$ is a neighborhood of $x$, for each $x \in X$. This is clear if we take into account that each element in $\B_x$ is an open set with respect to the topology $\tau$ (see Remark \ref{ob:base}). Indeed, if $x$ is left-isolated then, given $B \in \mathcal{B}_x$, we can write $B=[x,b[$, for some $b \in D$ with $b>x$. Equivalently, $B=]a,b[$, where $a$ is the previous point to $x$ according to the order. The other cases are similar.
			\item For each neighborhood of $x$, $U$, there exists $B \in \mathcal{B}_x$ such that $B \subseteq U$. Indeed, let $U$ be a neighborhood of $x$, then there exists an open set $G$ such that $x \in G \subseteq U$. Since $G$ is open and $\mathcal{B}=\{]a,b[:a<b\}$ is an open basis, we can consider $a, b$ such that $]a,b[ \subseteq G$ and $a<x<b$. Now we distinguish some cases depending on whether $x$ is isolated or not:
			
			\begin{itemize}
				\item[-] Suppose that $x$ is isolated, then there exist $y, z \in X$ such that $y<x<z$ and $]y,z[=\{x\}$. In this case $\{x\}$ is an element of $\B_x$ which is contained in $U$.
				\item[-] Suppose that $x$ is not left-isolated nor right-isolated. Since $]a,x[$ and $]x,b[$ are both open in $\tau$ and $D$ is dense in $\tau$, we can choose $c \in ]a,x[ \cap D$ and $d \in ]x,b[\cap D$. What is more, it holds that $x \in ]c,d[ \subseteq ]a,b[ \subseteq G \subseteq U$, what finishes the proof.
				\item[-] Suppose that $x$ is left-isolated but it is not right-isolated. Then there exists $y \in X$ such that $]y,x[=\emptyset$ and $]x,z[ \neq \emptyset$ for each $z>x$. Since $]x,b[$ is open in $\tau$ and $D$ is dense in $\tau$, we can choose $d \in ]x,b[\cap D$. What is more, it holds that $x \in [x,d[ \subseteq ]a,b[ \subseteq G \subseteq U$.
				\item[-] Suppose that $x$ is not letf-isolated but it is right-isolated. Then there exists $z \in X$ such that $]x,z[=\emptyset$ and $]y,x[ \neq \emptyset$ for each $y<x$. Since $]a,x[$ is a neighborhood in $\tau$ and $D$ is dense in $\tau$, we can choose $c \in ]a,x[\cap D$. What is more, it holds that $x \in ]c,x] \subseteq ]a,b[ \subseteq G \subseteq U$.
			\end{itemize}
		\end{itemize}
	\end{proof}
\end{pro}

We can choose a countable neighborhood basis of each point such that its elements are ordered, as next remark shows:

\begin{ob}
	Let $x \in X$, then there exists a  countable neighborhood basis of $x$, $\mathcal{B}_x'=\{]a_n',b_n'[: a_n'<x<b_n'; n \in \N\}$ such that $a_n$ is a non-decreasing sequence and $b_n$ a non-increasing one. 
	\begin{proof}
		
		Indeed, since $\tau$ is first countable, there exists a countable basis of each point. According to the previous proposition, in case that $x$ is not left-isolated nor right-isolated, we have that $\mathcal{B}_x=\{]a_1, b_1[: a_1<x<b_1; a_1 , b_1 \in D\}$ is a countable basis of $x$. Since $D$ is a dense subset in $\tau$ and $]x,b_1[$ and $]a_1,x[$ are nonempty open sets in $\tau$, there exists $d_{a_1} \in D \cap ]a_1,x[$ and $d_{b_1} \in ]x,b_1[ \cap D$. Now define $a_2=d_{a_1}$ and $b_2=d_{b_1}$. Moreover, there exists $d_{a_2} \in D \cap ]a_2, x[$ and $d_{b_2} \in D \cap ]x,b_2[$. Now we define $a_3=d_{a_2}$ and $b_3=d_{b_2}$. Recursively we have that $\mathcal{B}_x'=\{]a_n,b_n[: a_n<x<b_n; n \in \N, a_n, b_n \in D\}$ where $a_n=d_{a_{n-1}}$ and $b_n=d_{b_{n-1}}$. It is clear that $\mathcal{B}_x$ is a neighborhood basis of $x$. Moreover, given $n \in \N$ it holds that $]a_{n+1},b_{n+1}[ \subseteq ]a_n,b_n[$ by the definition of $a_n'$ and $b_n'$. We can proceed analogously to get basis for the right-isolated or left-isolated points. Moreover, note that if $x$ is isolated, the basis given in the previous proposition satisfies the condition given in this remark.
	\end{proof}
\end{ob}

There exists an equivalence between the property of second countable for $\tau$ and the countability of the set of isolated points.

\begin{pro}\label{pro:ANII}
	Let $X$ be a LOTS. $X$ is second countable with respect to the topology $\tau$ if, and only if $X$ is separable and the set of right-isolated or left-isolated points is countable.
	\begin{proof}
		
		Let us define $C_1$ and $C_2$ to be, respectively, the set of left-isolated points and the set of right-isolated points.
		
		$\Leftarrow)$ Let $D$ be a countable dense subset of $X$ and suppose that $C_1$ and $C_2$ are countable subsets. Consider the family $\B=\{\{x\}: x \in C_1 \cap C_2\} \cup \{]a,x]: a<x, x \in C_2, a \in D\}\cup \{[x,b[: x<b, x \in C_1, b \in D\} \cup \{]a,b[: a<b,a, b \in D\}$ and note that it is an open basis of $X$ with respect to $\tau$. What is more, the countability of the set of right-isolated and left-isolated points gives us that $\B$ is countable. Hence $\tau$ is second countable.
		
		$\Rightarrow)$ Suppose that $X$ is second countable with respect to $\tau$, then there exists a countable open basis, $\mathcal{B}=\{U_n: n \in \N\}$. Since second countable spaces are separable, we only have to prove that $C_1$ and $C_2$ are countable subsets, what gives us that $C_1 \cup C_2$ is also countable.
		
		\begin{itemize}
			\item $C_1$ is countable: let $x \in C_1$ and $b_1>x$ with $b_1 \in D$. Since $\mathcal{B}$ is an open basis and $[x,b_1[$ is an open set containing $x$, there exists $n_x \in \N$ such that $x \in U_{n_x} \subseteq [x,b_1[$. Now let $y \in C_1$ with $y \neq x$ and $b_2 \in D$ with $y<b_2$, then there exists $n_y \in \N$ such that $y \in U_{n_y} \subseteq [y, b_2[$ for $b_2>y$. Consequently, $f: C_1 \rightarrow \N$ given by $f(x)=n_x$ is an injective function, what proves the countability of $C_1$.
			
			\item The countability of $C_2$ can be proved similarly to the countability of $C_1$.
			\end{itemize}
	\end{proof}
\end{pro}

Now we define the concept of right convergent and left convergent sequence.

\begin{defi}
	Let $x \in \X$ and $\nu$ be a topology defined on $\X$. We say that a sequence $x_n \in \X$ is right $\nu$-convergent (respectively left $\nu$-convergent) to $x$ if $x_n \stackrel{\nu}{\rightarrow} x$ and $x_n \geq x$ (respectively $x_n \leq x$), for each $n \in \N$.
\end{defi}

Now we define the concept of monotonically right convergent and monotonically left convergent sequence.

\begin{defi}
	Let $x \in \X$ and $\nu$ be a topology defined on $\X$. We say that a sequence $x_n \in \X$ is monotonically right $\nu$-convergent (respectively monotonically left $\nu$-convergent) to $x$ if $x_n \stackrel{\nu}{\rightarrow} x$ and $x<x_{n+1} < x_{n}$ (respectively $x_n < x_{n+1} < x$), for each $n \in \N$.
\end{defi}

\begin{pro}\label{pro:isolated}
	Let $x \in X$. Then $x$ is not left-isolated (respectively right-isolated) if, and only if there exists a monotonically left $\tau$-convergent (respectively monotonically right $\tau$-convergent) to $x$ sequence.
	\begin{proof}
		$\Rightarrow)$ Let $x$ be a non left-isolated point, then $x \neq \min X$. Since $\tau$ is first countable (by Proposition \ref{pro:ANI}), we can consider a countable neighborhood basis of $x$, $\mathcal{B}_{x}=\{]a_n, b_n[: n \in \N\}$. Now let $a , b \in X$ be such that $a<x<b$, then there exists $n_1 \in \N$ such that $a \leq a_{n_1}<x$ due to the fact that $\mathcal{B}_x$ is a neighborhood basis of $x$. Since $x$ is not left-isolated, we can choose $z_1 \in ]a_{n_1},x[$. Now we can consider $n_2 \in \N$ such that $z_1 \leq a_{n_2} <x$ due to the fact that $\mathcal{B}_x$ is a neighborhood basis of $x$. Recursively, we can construct a subsequence of $a_n$, $a_{\sigma(n)}$, such that $a_{\sigma(n)}<a_{\sigma(n+1)}<x$ and $a_{\sigma(n)} \rightarrow x$, that is, $a_{\sigma(n)}$ is a monotonically left $\tau$-convergent sequence to $x$.
		
		The proof is analogous in case that $x$ is not right-isolated.
		
		$\Leftarrow)$ Let $x \in \X$ and suppose that it is a left-isolated point. If $x=\min X$ the proof is easy. Suppose that $x \neq \min X$, then there exists $z \in X$ such that $]z,x[= \emptyset$. Suppose that there exists a monotonically left $\tau$-convergent sequence to $x$, $x_n$, then it holds that there exists $n_0 \in \N$ such that $x_n >z$, for each $n \geq n_0$. Moreover, since $x_n<x$, we have that $x_n \in ]z,x[=\emptyset$, what is a contradiction. Hence, $x$ is not left-isolated.
		
		The case in which there exists a monotonically right $\tau$-convergent sequence to $x$ can be proven analogously.
	\end{proof}
\end{pro}

\begin{lema}\label{lema:infsup}
	\begin{enumerate}
		\item If $a_n$ is a monotonically left $\tau$-convergent sequence to $a$, then $\cup(<a_n)=(<a)=\cup(\leq a_n)$.
		\item If $a_n$ is a monotonically right $\tau$-convergent sequence to $a$, then $\cap(<a_n)=(\leq a)=\cap(\leq a_n)$.
	\end{enumerate}
	
	\begin{proof}
		
		\begin{enumerate}
			\item Next we prove both equalities:
			
			\begin{itemize}
				\item $\cup (< a_n)=(<a)$. On the one hand, since $a_n < a$, we have that $(<a_n) \subseteq (<a)$. Therefore, $\cup (< a_n) \subseteq (<a)$.
				
				On the other hand, let $x <a$. Since $a_n \stackrel{\tau}{\rightarrow} a$ and $a_n<a$, there exists $n \in \N$ such that $x <a_n<a$ and, hence, $x \in \cup(<a_n)$.
				
				\item $\cup(<a_n)=\cup (\leq a_n)$. On the one hand, let $x \in \cup(< a_n)$, then there exists $n \in \N$ such that $x \in (<a_n)$. It is clear that $x \in (\leq a_n)$ and, hence, $x \in \cup (\leq a_n)$.
				
				On the other hand, let $x \in \cup (\leq a_n)$, then there exists $n \in \N$ such that $x \in (\leq a_n)$. Since $a_n < a$ and $a_n \stackrel{\tau}{\rightarrow} a$, it holds that there exists $m >n$ such that $a_n <a_m<a$, The fact that $x \in (\leq a_n)$ gives us that $x \in (<a_m)$. We conclude that $x \in \cup (<a_n)$.
			\end{itemize}
			\item Next we prove both equalities:
			
			On the one hand, let $x \leq a_n$, for each $n \in \N$ and suppose that $x>a$, then there exists $m \in \N$ such that $a<a_m<x$, what is a contradiction with the fact that $x \leq a_n$, for each $n \in \N$. Hence, $x \leq a$ and $\cap (\leq a_n) \subseteq (\leq a)$. 
			
			Moreover, since $a<a_n$ for each $n \in \N$, we have that $(\leq a) \subseteq (<a_n)$. Therefore $(\leq a) \subseteq \cap(<a_n)$. 
			
			What is more, it is clear that $(<a_n) \subseteq (\leq a_n)$, so we conclude that $\cap (< a_n) \subseteq \cap (\leq a_n)$ and we finish the proof.
		\end{enumerate}
	\end{proof}
\end{lema}

\begin{pro}
	Each connected set in $\tau$ is convex. 
	\begin{proof}
		Let $A \subseteq \X$ be a connected set. Suppose that $A$ is not convex, what means that there exist $a, b \in A$ with $a<b$ such that there exists $x \in \X \setminus A$ with $a<x<b$. Note that $(<x)$ and $(>x)$ are both open sets in $\tau$, what implies that $U=(<x) \cap A$ and $V=(>x) \cap A$ are both open in $A$ with the topology induced by $\tau$ in $A$. Note that $U, V \neq \emptyset$ since $a \in U, b \in V$ and $U \cup V=A$ what implies that $A$ is not connected. Hence $A$ is convex.
	\end{proof}
\end{pro}

\section{Defining the distribution function}

The definition of the cumulative distribution function of a measure defined on the Borel $\sigma$-algebra of $\X$ is the next one:

\begin{defi}
	The cumulative distribution function (in short, cdf) of a probability measure $\mu$ is a function $F: \X \to [0,1]$ defined by $F(x)=\mu(\leq x)$. 
\end{defi}

\begin{lema}\label{lemaright}
	Let $\tau'$ be a first countable topology on $X$ such that $\tau \subseteq \tau'$. Let $f: X \rightarrow [0,1]$ be a monotonically non-decreasing function and $x \in \X$ and suppose that $f(x_n) \rightarrow f(x)$ for each monotonically right $\tau'$-convergent (respectively monotonically left $\tau'$-convergent) sequence to $x$, then $f$ is right $\tau'$-continuous (respectively $f$ is left $\tau'$-continuous).
	
	\begin{proof}
		Let $x \in \X$ and $x_n \stackrel{\tau'}{\rightarrow}x$ be a right $\tau'$-convergent sequence. If $(x_n)$ is eventually constant (there exists $k \in \N$ such that $x_n=x$ for each $n \geq k$), the proof is easy. Otherwise, using that $\tau \subseteq \tau'$, we can recursively define a decreasing subsequence $x_{\sigma(n)}$ of $x_n$, such that $x < x_{\sigma(n+1)} < x_{\sigma(n)}$, for each $n \in \N$.
		
		It follows that $x_{\sigma(n)}$ is monotonically right $\tau'$-convergent to $x$ and, hence, by hypothesis, $f(x_{\sigma(n)}) \rightarrow f(x)$. 
		
		Given $k \in \N$, we have that $x<x_{\sigma(k)}$. Since $\tau \subseteq \tau'$ it follows that $x_n \stackrel{\tau}{\rightarrow}x$ what gives us that there exists $n_0 \in \N$, such that $x \leq x_n <x_{\sigma(k)}$, for each $n \geq n_0$.

		Now, by the monotonicity of $f$, it follows that $f(x) \leq f(x_n) \leq f(x_{\sigma(k)})$, for each $n \geq n_0$. We conclude that $f(x_n) \rightarrow f(x)$ and, hence, $f$ is right $\tau'$-continuous.
		
		We can proceed analogously to show that $f$ is left $\tau'$-continuous when $x_n$ is left $\tau'$-convergent to $x$.
	\end{proof}
\end{lema}

\begin{cor}\label{lemacont}
	Let $\tau'$ be a first countable topology on $X$ and $f:X \rightarrow [0,1]$ a function. If $f$ is right and left $\tau'$-continuous, then $f$ is $\tau'$-continuous. 
	\begin{proof}
		Let $x \in \X$ and $x_n \stackrel{\tau'}{\rightarrow} x$. Let $\sigma_1,\sigma_2: \N \to \N$ be two increasing functions such that $x_{\sigma_1(n)} \geq x$ and $x_{\sigma_2(n)} \leq x$ with $\sigma_1(\N) \cup \sigma_2(\N)=\N$. If either $\sigma_1(\N)$ or $\sigma_2(\N)$ is finite then the proof is easy. Otherwise, $(x_{\sigma_1(n)})$ is a right subsequence of $(x_n)$ and $(x_{\sigma_2(n)})$ is a left subsequence of $(x_n)$. By hypothesis, it holds that $f(x_{\sigma_1(n)}) \rightarrow f(x)$ and $f(x_{\sigma_2(n)}) \rightarrow f(x)$. It easily follows that $f(x_n) \rightarrow f(x)$, what means that $f$ is continuous with respect to the topology $\tau'$.
	\end{proof}
\end{cor}

\begin{ob}\label{ob:right}
	Note that Lemma \ref{lemaright} and Corollary \ref{lemacont} can be both applied to topology $\tau$.
\end{ob}

\begin{cor}\label{lemaconttauo}
	Let $\tau'$ be a first countable topology on $X$ with $\tau \subseteq \tau'$ and let $f:X \rightarrow [0,1]$ be a monotonically non-decreasing function. Suppose that $f(x_n) \rightarrow f(x)$ for each monotonically right $\tau'$-convergent sequence to $x$ and each monotonically left $\tau'$-convergent sequence, $x_n$, to $x$, then $f$ is continuous (with respect to the topology $\tau'$).
	\begin{proof}
		It follows from Lemma \ref{lemaright} and Corollary \ref{lemacont}.
	\end{proof}
\end{cor}

\begin{pro}\label{propcdf}
	Let $F$ be a cdf. Then:
	\begin{enumerate}
		\item $F$ is monotonically non-decreasing.
		\item $F$ is right $\tau$-continuous.
		\item If there does not exist $\min X$, then $\inf F(X)=0$.
		\item $\sup F(X)=1$. 
	\end{enumerate}
	
	\begin{proof}
		\begin{enumerate}
			\item This is obvious if we take into account the monotonicity of $\mu$ that follows from the fact that $\mu$ is a measure.
			
			\item For the purpose of proving that $F$ is right $\tau$-continuous, let $x_n$ be a monotonically right $\tau$-convergent sequence to $x$. Let us see that $F(x_n) \rightarrow F(x)$. 
			
			First of all, note that the fact that $x_n$ is a monotonically right $\tau$-convergent sequence to $x$ implies, by Lemma \ref{lema:infsup}.2, that $\bigcap_n(\leq x_n)=(\leq x)$. Moreover, $(\leq x_n)$ is a monotonically non-increasing sequence so $(\leq x_n) \rightarrow \bigcap_n (\leq x_n)=(\leq x)$. Thus, from the continuity of the measure $\mu$, it follows that $\mu(\leq x_n) \rightarrow \mu(\leq x)$, that is, $F(x_n) \rightarrow F(x)$. Therefore, by Lemma \ref{lemaright} and Remark \ref{ob:right}, we have that $F$ is right $\tau$-continuous. 
			
			\item Suppose that there does not exist $\min X$. By Proposition \ref{pro:leftboundseq}, there exists a sequence $x_n$ in $X$ such that $x_{n+1}<x_n$, for each $n \in \N$ and $\{x_n: n \in \N\}^l=X^l=\emptyset$. Then we have $\bigcap (\leq x_n)=\emptyset$. Now, note that $(\leq x_n)$ is a monotonically non-increasing sequence, what implies that $(\leq x_n) \rightarrow \bigcap (\leq x_n)=\emptyset$. By the continuity of the measure $\mu$ it holds that $\mu(\leq x_n)=F(x_n) \rightarrow \mu(\emptyset)=0$. Hence $\inf \{F(x_n): n \in \N\}=0$. Finally, if we join the previous equality with the fact that $0 \leq \inf F(X) \leq \inf \{F(x_n): n \in \N\}$, we conclude that $\inf F(X)=0$. 
			
			\item We distinguish two cases depending on whether there exists the maximum of $X$ or not:
			
			\begin{enumerate}
				\item Suppose that there exists $\max \X$. In this case $\sup F(X)=F(\max X)=\mu(\X)=1$.
				
				\item Suppose that there does not exist $\max \X$. By Proposition \ref{pro:leftboundseq}, there exists a sequence $x_n$ in $X$ such that $x_{n+1}>x_n$, for each $n \in \N$ and $\{x_n: n \in \N\}^u=X^u=\emptyset$. Then we have $\bigcup (\leq x_n)=X$. Now, note that $(\leq x_n)$ is a monotonically non-decreasing sequence, what implies that $(\leq x_n) \rightarrow \bigcup (\leq x_n)=\X$. By the continuity of the measure $\mu$ it holds that $\mu(\leq x_n)=F(x_n) \rightarrow \mu(\X)=1$. Hence $\sup \{F(x_n): n \in \N\}=1$. Finally, if we join the previous equality with the fact that $\sup \{F(x_n): n \in \N\} \leq \sup F(X) \leq 1$, we conclude that $\sup F(X)=1$.
			\end{enumerate}
		\end{enumerate}
	\end{proof}
\end{pro}

The previous proposition makes us wonder the next

\begin{que}
	Let $F: X \rightarrow [0,1]$ be a function satisfying the properties collected in Proposition \ref{propcdf}, does there exist a probability measure $\mu$ on $X$ such that its cdf, $F_\mu$, is $F$?
\end{que}

According to the previous results we can conclude that

\begin{cor}\label{cor:continuity}
	Let $F$ be a cdf and $x \in X$. Then $F$ is $\tau$-continuous at $x$ if, and only if $F$ is left $\tau$-continuous at $x$.
\end{cor}

\begin{pro}\label{pro:isolated-cont}
	Let $x \in \X$ and $f$ be a monotonically non-decreasing function. If $x$ is left-isolated (respectively right-isolated), then $f$ is left $\tau'$-continuous (respectively right $\tau'$-continuous) where $\tau'$ is a first countable topology such that $\tau \subseteq \tau'$.
	\begin{proof}
		Let $x \in \X$ and suppose that it is left-isolated. The case in which $x=\min X$ is obvious. Suppose that $x \neq \min X$, then there exists $z \in X$ such that $]z,x[=\emptyset$. Hence $(>z)$ is open in $\tau$ and, consequently, a neighbourhood of $x$. Let $x_n$ be a left $\tau'$-convergent sequence to $x$, then it is also left $\tau$-convergent to $x$. Hence, there exists $n_0 \in \N$ such that $x_n \in (>z)$, for each $n \geq n_0$. Since $x_n \leq x$, we have that $x_n=x$, for each $n \geq n_0$. Consequently, $f(x_n) \rightarrow f(x)$ and $f$ is left $\tau'$-continuous.

		The case in which $x$ is right-isolated can proved analogously.
	\end{proof}
\end{pro}

\begin{cor}\label{isolated-cont}
	Let $\mu$ be a probability measure on $\X$ and $F$ its cdf. Let $x \in \X$. If $x$ is left-isolated, then $F$ is $\tau$-continuous at $x$.
	\begin{proof}
		It immediately follows from Proposition \ref{pro:isolated-cont}, Corollary \ref{cor:continuity} and Proposition \ref{propcdf}.
	\end{proof}
\end{cor}

\begin{defi}
	Let $\mu$ be a probability measure on $\X$ and $F$ its cdf. We define $F_-: \X \rightarrow [0,1]$, by $F_-(x)=\mu(<x)$, for each $x \in \X$.
\end{defi}

Note that $F_-$ is monotonically non-decreasing by the monotonicity of the measure.

Next we introduce two results which relate $F_-$ with $F$.
\begin{pro}\label{pro:sup}
	Let $\mu$ be a probability measure on $\X$ and $F$ its cdf. Then $\sup F(<x)=F_-(x)$, for each $x \in \X$ with $x \neq \min X$.
	\begin{proof}
		$\geq)$ Let $x \in \X$ with $x \neq \min X$. We distinguish two cases depending on whether $x$ is left-isolated or not:
		\begin{enumerate}
			\item Suppose that $x$ is not left-isolated, then by Proposition \ref{pro:isolated}, there exists a monotonically left $\tau$-convergent sequence, $a_n$, to $x$. This implies that $(\leq a_n) \rightarrow \cup (\leq a_n)$. Moreover, Lemma \ref{lema:infsup}.3 gives us that $\cup (\leq a_n)=(<x)=\cup (< a_n)$. Hence, $(\leq a_n) \rightarrow (<x)$ and, consequently, $F(a_n) \rightarrow \mu(<x)$. Now, since $a_n<x$, $F(a_n) \leq \sup F(<x)$. If we take limits, we have that $\mu(<x)=F_-(x) \leq \sup F(<x)$.
			
			\item Suppose that $x$ is left-isolated, then there exists $z \in \X$ such that $z<x$ and $]z,x[=\emptyset$, what implies that $F(z)\leq \sup F(<x)$. Moreover, note that $(<x)=(\leq z)$ what means that $\mu(<x)=F(z)$. We conclude that $\mu(<x)=F_-(x) \leq \sup F(<x)$.
		\end{enumerate}

		$\leq)$ Let $y \in \X$ with $y<x$, then $F(y) \leq \mu(<x)$ and hence $\sup F(<x) \leq \mu(<x)=F_-(x)$.

	\end{proof}
\end{pro}

We can recover the cdf $F$ from $F_-$.

\begin{pro}
	Let $F$ be a cdf, then $F(x)=\inf F_-(>x)$, for each $x \in X$ with $x \neq \max X$.
	\begin{proof}
		$\leq)$ Let $x \in X$ with $x \neq \max X$ and $y \in X$ be such that $y>x$, then $\mu(<y) \geq \mu(\leq x)$, that is, $F(x) \leq F_-(y)$, what gives us that $F(x) \leq \inf F_-(>x)$.
		
		$\geq)$ Let $x \in X$ with $x \neq \max X$. We distinguish two cases depending on whether $x$ is right-isolated or not.
		
		\begin{enumerate}
			\item Suppose that $x$ is right-isolated, then there exists $z \in X$ such that $z>x$ and $]x,z[=\emptyset$, what implies that $\inf F_-(>x) \leq F_-(z)$. Moreover, note that $(>x)=(\geq z)$ what means that $\mu(>x)=\mu(\geq z)$ or, equivalently, $\mu(\leq x)=\mu(<z)$. Hence, $F(x)=F_-(z)$. We conclude that $\inf F_-(>x)\leq F(x)$.
			
			\item Suppose that $x$ is not right-isolated, then by Proposition \ref{pro:isolated}, there exists a monotonically right $\tau$-convergent sequence, $a_n$, to $x$. Since $F$ is right $\tau$-continuous, we have that $F(a_n) \rightarrow F(x)$. Now, the fact that $a_n>x$ gives us that $\inf F_-(>x) \leq F_-(a_n) \leq F(a_n)$. Finally, if we take limits, we have that $\inf F_-(>x) \leq F(x)$.
			
		\end{enumerate}
	\end{proof}
\end{pro}

\begin{lema}\label{lema:medidapunto}
	Let $\mu$ be a probability measure on $\X$ and $F$ its cdf. Given $x \in \X$, it holds that $F(x)=F_-(x)+\mu(\{x\})$.
	\begin{proof}
		Indeed, given $x \in \X$, by the definition of cdf, we have that $F(x)=\mu(\leq x)$. Now, since $\mu$ is $\sigma$-additive, $F(x)=\mu(<x)+\mu(\{x\})$. We conclude that $F(x)=F_-(x)+\mu(\{x\})$.
	\end{proof}
\end{lema}

A cdf let us calculate the measure of $]a,b]$ for each $a \leq b$ according to the next proposition and Lemma \ref{lema:medidapunto}.

\begin{pro}\label{pro:measureab}
	Let $\mu$ be a probability measure on $\X$ and $F$ its cdf, then $\mu(]a,b])=F(b)-F(a)$ for each $a, b \in \X$ with $a < b$.
	\begin{proof}
		Note that we can write $(\leq b)= (\leq a) \cup ]a,b]$. Now, since $\mu$ is a measure (and hence $\sigma$-additive) it holds that  $\mu(\leq b)= \mu(\leq a) + \mu(]a,b])$, that is, $\mu(]a,b])=F(b)-F(a)$.
	\end{proof}
\end{pro}

\begin{cor}\label{cor:measure} Let $\mu$ be a probability measure on $\X$ and $F$ its cdf, then:
	\begin{enumerate} 
		\item $\mu([a, b])=F(b)-F_-(a)$.
		\item $\mu(]a,b[)=F_-(b)-F(a)$.
		\item $\mu([a,b[)=F_-(b)-F_-(a)$.
	\end{enumerate}
	
	\begin{proof}
		The proof is immediate if we take into account the previous proposition and Lemma \ref{lema:medidapunto}.
		
	\end{proof}
\end{cor}

\begin{pro}\label{pro:FconvergesF-}
	Let $\mu$ be a probability measure on $\X$ and $F$ its cdf. Let $x \in \X$ and $x_n$ be a monotonically left $\tau$-convergent sequence to $x$ then $F(x_n) \rightarrow F_-(x)$.
	
	\begin{proof}
		Let $x \in \X$ and $x_n$ be a monotonically left $\tau$-convergent sequence to $x$. Lemma \ref{lema:infsup}.5 gives us that $\bigcup_n(\leq x_n)=(<x)$. Note that $(\leq x_n)$ is a monotonically non-decreasing sequence, what means that $(\leq x_n) \rightarrow \bigcup_n(\leq x_n)=(<x)$. Finally, by the continuity of $\mu$ it follows that $\mu(\leq x_n) \rightarrow \mu(<x)=F_-(x)$, that is, $F(x_n) \rightarrow F_-(x)$.
	\end{proof}
\end{pro}

Next we collect the properties of $F_-$:

\begin{pro}\label{contizq-tauo}
	
	Let $\mu$ be a probability measure on $\X$ and $F$ its cdf, then:
	\begin{enumerate}
		\item $F_-$ is monotonically non-decreasing.
		\item $F_-$ is left $\tau$-continuous. 
		\item $\inf F_-(X)=0$.
		\item If there does not exist the maximum of $X$, then $\sup F_-(X)=1$. Otherwise, $F_-(\max X)=1-\mu(\{\max X\})$.
	\end{enumerate}
	
	\begin{proof}
		\begin{enumerate}
			\item This is obvious if we take into account the monotonicity of $\mu$ that follows from the fact that it is a measure.
			
			\item Let $x_n$ be a monotonically left $\tau$-convergent sequence to $x$, then by Proposition \ref{pro:FconvergesF-}, it holds that $F(x_n) \rightarrow F_-(x)$. Since $x_n$ is monotonically left $\tau$-convergent, it holds that $x_n<x_{n+1} <x$, so the fact that $F_-$ is monotonically non-decreasing implies that $F(x_n) \leq F_-(x_{n+1})\leq F_-(x)$. By taking limits, we conclude that $F_-(x_n) \rightarrow F_-(x)$, and by Lemma \ref{lemaright} and Remark \ref{ob:right}, $F_-$ is left $\tau$-continuous. 
			
			\item By Proposition \ref{pro:leftboundseq}, there exists a sequence $x_n$ in $X$ such that $x_{n+1}<x_n$, for each $n \in \N$ and $\{x_n: n \in \N\}^l=X^l=\emptyset$. Then we have $\bigcap (< x_n)=\emptyset$. Now, note that $(< x_n)$ is a monotonically non-increasing sequence, what implies that $(< x_n) \rightarrow \bigcap (< x_n)=\emptyset$. By the continuity of the measure $\mu$ it holds that $\mu(< x_n)=F_-(x_n) \rightarrow \mu(\emptyset)=0$. Hence $\inf \{F_-(x_n): n \in \N\}=0$. Finally, if we join the previous equality with the fact that $0 \leq \inf F_-(X) \leq \inf \{F_-(x_n): n \in \N\}$, we conclude that $\inf F_-(X)=0$.
			
			\item We distinguish two cases depending on whether there exists the maximum of $X$ or not:
			
			\begin{enumerate}
				\item Suppose that there does not exist $\max \X$. By Proposition \ref{pro:leftboundseq}, there exists a sequence $x_n$ in $X$ such that $x_{n+1}>x_n$, for each $n \in \N$ and $\{x_n: n \in \N\}^u=X^u=\emptyset$. Then we have $\bigcup (< x_n)=X$. Now, note that $(< x_n)$ is a monotonically non-decreasing sequence, what implies that $(< x_n) \rightarrow \bigcup (< x_n)=\X$. By the continuity of the measure $\mu$ it holds that $\mu(< x_n)=F_-(x_n) \rightarrow \mu(\X)=1$. Hence $\sup \{F_-(x_n): n \in \N\}=1$. Finally, if we join the previous equality with the fact that $\sup \{F_-(x_n): n \in \N\} \leq \sup F_-(X) \leq 1$, we conclude that $\sup F_-(X)=1$.
				\item Now suppose that there exists $\max X$, then Lemma \ref{lema:medidapunto} let us claim that $F_-(\max X)=F(\max X)-\mu(\{\max X\})=1-\mu(\{\max X\})$.
			\end{enumerate}

		\end{enumerate}
	\end{proof}
\end{pro}

\section{Discontinuities  of a cdf}

In this section we prove some results which are analogous to those proven in Chapter 1 of \cite{chung} and which are related to the discontinuities of a cdf.

First of all, we give a sufficient condition to ensure that a cdf is continuous at a point.

\begin{pro}\label{pro:cont-munula}
	Let $x \in \X$, $\mu$ be a probability measure on $\X$ and $F$ its cdf. If $\mu(\{x\})=0$ then $F$ is $\tau$-continuous at $x$.
	
	\begin{proof}
		Let $x_n$ be a monotonically left $\tau$-convergent sequence to $x$, then by Proposition \ref{pro:FconvergesF-}, it holds that $F(x_n) \rightarrow F_-(x)$. By Lemma \ref{lema:medidapunto}, it holds that $F(x)=F_-(x)$, so $F(x_n) \rightarrow F(x)$, and by Lemma \ref{lemaright} and Remark \ref{ob:right}, $F$ is left $\tau$-continuous. Finally, by Corollary \ref{cor:continuity}, $F$ is $\tau$-continuous.
		
	\end{proof}
\end{pro}

Next we introduce a lemma that will be crucial to show that the set of discontinuity points of a cdf is at most countable.

\begin{lema}
	Let $\mu$ be a probability measure on $\X$ and $F$ its cdf. Then $\{x \in \X: \mu(\{x\})>0\}$ is countable.
	\begin{proof}
		For every integer $N$, the number of points satisfying $\mu(\{x\})>\frac{1}{N}$ is, at most, $N$. Hence, there are no more than a countable number of points with positive measure.
	\end{proof}
\end{lema}

Next we collect two properties of a cdf $F_\mu$.

\begin{pro} Let $\mu$ be a probability measure on $\X$, then
	\begin{enumerate}
		\item $F_\mu$ is determined by a dense set, $D$, in $\X$ (with respect to the topology $\tau$) in its points with null measure,  that is, if for each $x \in D$ it holds that $F_\mu(x)=F_\delta(x)$, then $F_\mu(x)=F_\delta(x)$, for each $x \in \X$ with $\mu(\{x\})=0$ and $\delta(\{x\})=0$, where $F_\delta$ is the cdf of a probability measure, $\delta$, on $X$.
		
		\item The set of discontinuity points of $F_\mu$ with respect to the topology $\tau$ is countable.
	\end{enumerate}
	
	\begin{proof}
		\begin{enumerate}
			\item Let $x \in \X$ with $\mu(\{x\})=0$ and $\delta(\{x\})=0$. We distinguish two cases:
			
			\begin{itemize}
				\item Suppose that $x$ is left-isolated and right-isolated, then there exist $y, z \in X$ such that $]y,z[=\{x\}$, what implies that $x \in D$ due to the fact that $D$ is dense. Consequently, $F_\mu(x)=F_\delta(x)$.
				\item $x$ is not left-isolated or it is not right-isolated. If $x$ is not left-isolated, by Proposition \ref{pro:isolated}, there exists a sequence $x_n \stackrel{\tau}{\rightarrow} x$ such that $x_n< x_{n+1}<x$. Now, since $D$ is dense, it follows that there exists $d_n \in D$ such that $x_n<d_n<x_{n+1}$ and hence $d_n < d_{n+1}$, for each $n \in \N$. Hence, $d_n \rightarrow x$ in $\tau$. By hypothesis, we have that $F_\mu(d_n)=F_\delta(d_n)$. By Proposition \ref{pro:FconvergesF-}, $F_\mu(d_n) \rightarrow F_{\mu-}(x)$. But $F_{\mu-}(x)=F_{\mu}(x)$ since $\mu(\{x\})=0$ by Lemma \ref{lema:medidapunto}. Analogously, $F_\delta(d_n) \rightarrow F_\delta(x)$. Consequently, $F_{\mu}(x)=F_{\delta}(x)$.
				
				The case in which $x$ is not right-isolated can be proved analogously.
			\end{itemize}

			\item Let $x \in \X$. By Proposition \ref{pro:cont-munula}, we know that the fact that $F_\mu$ is not continuous at $x$ means that $\mu(\{x\})>0$. Since, by previous lemma, we have that $\{x \in \X: \mu(\{x\})>0\}$ is countable, we conclude that the set of discontinuity points is at most countable too.
		\end{enumerate}
	\end{proof}
\end{pro}

\section{The inverse of a cdf}

In this section, we see how to define the pseudo-inverse of a cdf $F$ defined on $\X$ and we gather some properties which relate this function to both $F$ and $F_-$. Its properties are similar to those which characterizes the pseudo-inverse in the classical case (see, for example, \cite[Th. 1.2.5]{Baz}). Moreover, we see that it is measurable.

Now, we recall the definition of this function in the classical case (see Section 1) so as to give a similar one in the context of a linearly ordered topological space. However, there exists a problem when we mention the infimum of a set, since it is not true that every set has infimum. Hence, we restric that definition to those points which let us talk about the infimum of a set as next definition shows.

\begin{defi}
	Let $F$ be a cdf. We define the pseudo-inverse of $F$ as $G:[0,1] \rightarrow X$ given by $G(x)=\inf \{y \in \X: F(y) \geq x\}$ for each $x \in [0,1]$ such that there exists the infimum of $\{y \in \X: F(y) \geq x\}$. 
\end{defi}

According to the previous definition, it is clear that

\begin{pro}\label{pro:mon}
	$G$ is monotonically non-decreasing.
	
	\begin{proof}
		Let $x, y \in [0,1]$ with $x<y$. Note that $\{z \in \X: F(z) \geq y\} \subseteq \{z \in \X: F(z) \geq x\}$ and it follows that $\inf\{z \in \X: F(z) \geq x\} \leq \inf\{z \in \X: F(z) \geq y\}$, that is, $G(x) \leq G(y)$, what means that $G$ is monotonically non-decreasing.
	\end{proof}
\end{pro}

\begin{lema}\label{lema:sucdec}
	Let $a=\inf \{a_n: n \in \N\}$ (respectively $a=\sup \{a_n: n\in \N\}$) where $a_n$ is a sequence such that $a_{n+1}<a_n$ (respectively $a_{n+1}>a_n$), for each $n \in \N$. Then $a_n \stackrel{\tau}{\rightarrow} a$.
	\begin{proof}
		Let $a_n$ be a sequence in $X$ such that $a_{n+1}<a_n$, for each $n \in \N$ and suppose that there exists $a=\inf \{a_n:n \in \N\}$. Let $b, c \in X \cup \{-\infty,\infty\}$ be such that $b<a<c$. Suppose that $a_n \geq c$, for each $n \in \N$, then $\inf \{a_n: n \in \N\}\geq c>a$, a contradiction with the fact that $a=\inf\{a_n: n\in \N\}$. Hence, there exists $n_0 \in \N$ such that $a_{n_0}<c$. What is more, $a_n <c$, for each $n \geq n_0$ since $a_{n+1} <a_n$, for each $n \in \N$. Consequently, $a_n \stackrel{\tau}{\rightarrow} a$. 
		
		The case in which $a=\sup \{a_n: n\in \N\}$ and $a_{n+1}>a_n$ can be proven analogously.
	\end{proof}
\end{lema}

Hereinafter, when we apply $G$ to a point, we assume that $G$ is defined in that point.

\begin{pro}\label{pro:inverse}
	Let $F$ be a cdf. Then:
	\begin{enumerate}
		\item $G(F(x)) \leq x$, for each $x \in X$.
		\item $F(G(r)) \geq r$, for each $r \in [0,1]$.
	\end{enumerate}
	\begin{proof}
		\begin{enumerate}
			\item Indeed, $x \in \{z \in \X: F(z) \geq F(x)\}$, and hence $\inf\{z \in \X: F(z) \geq F(x)\} \leq x$, which is equivalent to $G(F(x)) \leq x$. This proves the first item.
			\item Now let $y = G(r)=\inf\{z \in \X: F(z) \geq r\}$. If $y=\min\{z \in \X: F(z) \geq r\}$, it is clear that $F(y) \geq r$. Suppose that $y \neq \min \{z \in \X: F(z) \geq r\}$, then by Proposition \ref{pro:leftboundseq} there exists a sequence $y_n \in \{z \in \X: F(z) \geq r\}$ such that $y_{n+1}<y_n$ and $\{y_n: n \in \N\}^l=\{z \in \X: F(z) \geq r\}^l$. What is more, by Lemma \ref{lema:infA}, it holds that $y=\inf\{y_n: n \in \N\}$. Hence, Lemma \ref{lema:sucdec} let us claim that $y_n \stackrel{\tau}{\rightarrow} y$. Consequently, the right $\tau$-continuity of $F$ gives us that $F(y_n) \rightarrow F(y)$. Moreover, $F(y_n) \geq r$ since $y_n \in \{z \in \X: F(z) \geq r\}$. If we join this fact with the fact that $F(y_n) \rightarrow F(y)$, we conclude that $F(y) \geq r$. This proves the second item.
		\end{enumerate}
	
	\end{proof}
\end{pro}

We get, as an immediate corollary, that

\begin{cor}\label{cor:inverse}
	$G(r) \leq x$ if, and only if $r \leq F(x)$, for each $x \in X$ and each $r \in [0,1]$.
\end{cor}

Next result collects some properties of $G$ which arise from some relationships between $F$ and $F_-$ and some conditions on them.

\begin{pro}\label{pro:inverse:F}
	Let $F$ be a cdf and let $x \in \X$ and $r \in [0,1]$. Then:
	\begin{enumerate}
		\item $F(x)<r$ if, and only if $G(r)>x$.
		\item If $F_-(x)<r$, then $x \leq G(r)$.
		\item If $F_-(x) <r \leq F(x)$, then $G$ is defined in $r$ and $G(r)=x$.
		\item If $r < F_-(x)$, then $G(r)<x$.
		\item If $r=F_-(x)$, then $G(r) \leq x$.
	\end{enumerate}
	\begin{proof}
		\begin{enumerate}
			\item Note that it is an immediate consequence of Corollary \ref{cor:inverse}.
			\item Suppose that $G(r)<x$, then $\mu(<x) \geq \mu(\leq G(r))$ or, equivalently, $F_-(x) \geq F(G(r)) \geq r$, that is, $F_-(x) \geq r$.
			\item Let $x \in \X$ and $r \in [0,1]$ be such that $F_-(x)<r \leq F(x)$. First, note that if $y<x$, then $F(y) \leq \sup F(<x)=F_-(x)<r$ and hence $x=\inf\{y \in X:F(y) \geq r\}$. It follows that $G$ is defined in $r$ and $x=G(r)$.
			\item Let $x \in \X$ and $r \in [0,1]$. Suppose that $r<F_-(x)$. Since $F_-(x)=\sup F(<x)$, there exists $y <x$ such that $r<F(y)\leq F_-(x)$. Since $F(y)>r$, then $y \geq \inf\{z \in \X: F(z) \geq r\} = G(r)$. We conclude that $G(r)<x$. 
			
			\item Suppose that $r=F_-(x)$. The fact that $F_-(x) \leq F(x)$, for each $x \in \X$ gives us that $F(x) \geq r$, which is equivalent, by Corollary \ref{cor:inverse}, to $G(r) \leq x$.
		\end{enumerate}
	\end{proof}
\end{pro}

We prove another property of $G$.
\begin{pro}\label{inverse}
	$G$ is left $\tau$-continuous.
	\begin{proof}
	
			Let $(r_n)$ be a sequence in $[0,1[$ which is left convergent to $r \in [0,1[$ with $r_n \neq r$. Since $r_n \leq r$, by the monotonicity of $G$ (see Proposition \ref{pro:mon}) we have that $G(r_n) \leq G(r)$. Now we prove that $G(r)=\sup \{G(r_n): n \in \N\}$. For this purpose, let $x \in \{G(r_n): n \in \N\}^u$ and suppose that $x<G(r)$. By Proposition \ref{pro:inverse:F}.1, it holds that $F(x)<r$, so there exists $n \in \N$ such that $F(x)<r_n$. On the other hand, since $x \in \{G(r_n): n \in \N\}^u$ then $G(r_n) \leq x$, for each $n \in \N$. By the monotonicity of $F$ we have that $F(G(r_n)) \leq F(x)$ and, hence, by Proposition \ref{pro:inverse}.1, $r_n \leq F(x)$ since $F(G(r_n)) \geq r_n$. If we join this fact with the fact that $F(x)<r_n$, for some $n \in \N$, we conclude that $r_n<r_n$, a contradiction.
			
			It follows, by Lemma \ref{lema:sucdec}, that $(G(r_n))$ $\tau$-converges to $G(r)$.

	\end{proof}
\end{pro}

Next proposition collects some properties of $F$ and $F_-$ which arise from considering some conditions on $G$.

\begin{pro}\label{pro:inverse:G}
	Let $F$ be a cdf and let $x \in \X$ and $r \in [0,1]$. Then:
	\begin{enumerate}
		\item $G(r)>x$ if, and only if $F(x)<r$.
		\item If $G(r)=x$, then $F_-(x) \leq r \leq F(x)$.
		\item If $G(r)<x$, then $r \leq F_-(x)$.
	\end{enumerate}
	
	\begin{proof}
		\begin{enumerate}
			\item Note that this item is the same as the first item of Proposition \ref{pro:inverse:F}.
			\item Suppose that $G(r)=x$ and that $r >F(x)$, by item 1 it follows that $G(r)>x$, what is a contradiction with the fact that $G(r)=x$.
			
			Now suppose that $r <F_-(x)$, then item 4 of Proposition \ref{pro:inverse:F} gives us that $G(r)<x$, what is a contradiction with the fact that $G(r)=x$. 
			
			We conclude that $F_-(x) \leq r \leq F(x)$.
			\item It is equivalent to Proposition \ref{pro:inverse:F}.2.
		\end{enumerate}
	\end{proof}
\end{pro}

Some consequences that arise from the previous propositions are collected next.

\begin{cor}
	Let $F$ be a cdf and $r \in [0,1]$. Then:
	\begin{enumerate}
		\item $F_-(G(r))\leq r  \leq F(G(r))$.
		\item If $F(G(r)) >r$, then $\mu(\{G(r)\})>0$.
	\end{enumerate}
	
	\begin{proof}
		\begin{enumerate}
			\item Let $r \in [0,1]$. On the one hand, suppose that $F_-(G(r))>r$, then, by item 4 of Proposition \ref{pro:inverse:F}, it holds that $G(r)<G(r)$, what is a contradiction. Hence, $F_-(G(r)) \leq r$.
			
			On the other hand, the inequality $r \leq F(G(r))$ is clear if we take into account Proposition \ref{pro:inverse}.
			
			\item By Lemma \ref{lema:medidapunto} $F(x)=F_-(x)+\mu(\{x\})$, for each $x \in \X$, so we have that $F(G(r))=F_-(G(r))+\mu(\{G(r)\})$. If $F(G(r))>r$, it holds that $F_-(G(r))+\mu(\{G(r)\})>r$. Moreover if we join this fact with the previous item, we conclude that $\mu(\{G(r)\})>0$. 
		\end{enumerate}
	\end{proof}
\end{cor}

\begin{cor} \label{cor:fgrr}
	Let $r \in [0,1]$. If $\mu(\{G(r)\})=0$, then $F(G(r))=r$.
\end{cor}

Now, we introduce some results in order to characterize the injectivity of $G$ and $F$.

\begin{pro} \label{pro:giny}
	$\mu(\{x\})=0$, for each $x \in \X$ if, and only if $G$ is injective.
	\begin{proof}
		$\Rightarrow)$ It immediately follows from the second item of Proposition \ref{pro:inverse:G}. Indeed, this proposition gives us that if $G(r)=x$, then $F_-(x) \leq r \leq F(x)$. Suppose that there exists $r, s \in X$ such that $r \neq s$ with $G(r)=G(s)=x$, then $F_-(G(r)) \leq r \leq F(G(r))$ and  $F_-(G(r)) \leq s \leq F(G(r))$. Since $\mu(\{G(r)\})=0$, it holds that $F_-(G(r))=F(G(r))=r=s$, an hence $G$ is injective.
		
		$\Leftarrow)$ Suppose that there exists $x \in \X$ such that $\mu(\{x\})>0$, then $F_-(x)<F(x)$. Now let $r \in [0,1]$ be such that $F_-(x)<r<F(x)$. By Proposition \ref{pro:inverse:F}.3 we have that $G$ is defined in $r$ and $G(r)=G(F(x))=x$, for each $r \in ]F_-(x), F(x)[$, what is a contradiction with the fact that $G$ is injective.
	\end{proof}
\end{pro}

\begin{pro} \label{pro:finy}
	Let $F$ be a cdf, then $F$ is injective if, and only if $\mu(]a,b])>0$, for each $a<b$.
	
	\begin{proof}
		Let $a, b \in X$ be such that $a<b$. Note that, by Proposition \ref{pro:measureab}, $\mu(]a,b])=0$ is equivalent to $F(b)-F(a)=0$, that is, $F(b)=F(a)$ if, and only if $F$ is not injective.
	\end{proof}
\end{pro}

And we get, as immediate corollary, the next one

\begin{cor}
	Let $F$ be a cdf of a probability measure $\mu$, and let $A \subseteq [0,1]$ be the subset of points where $G$ is defined. The following statements are equivalent:
	\begin{enumerate}
		\item $F \circ G(r)=r$ for each $r \in A$, $F(X) \subseteq A$ and  $G \circ F(x)=x$ for each $x \in X$.
		\item $F$ is injective and $F(X)=A$.
		\item $G:A \to X$ is bijective.
		\item $\mu(]a,b])>0$, for each $a<b$ and $\mu(\{a\})=0$, for each $a \in X$.
	\end{enumerate}
\end{cor}

\begin{proof}
First, we prove the following

\textbf{Claim}. \textit{If $F$ is injective, then $F(X) \subseteq A$ and $G(F(x))=x$ for each $x \in X$.}

Suppose that there exists $x \in X$ such that $G$ is not defined in $F(x)$, that is, there does not exists the infimum of $\{y \in X:F(y) \geq F(x)\}$. It follows that $x$ is not the infimum of the latter set, so there exists $y<x$ with $F(y) \geq F(x)$. By monotonicity of $F$ it follows that $F(y)=F(x)$, and since $F$ is injective $y=x$, a contradiction. We conclude that $F(X) \subseteq A$.

Finally, let $x \in X$, then $F(x) \in A$ and $G(F(x))=\inf \{y \in X:F(y) \geq F(x)\}$. On the other hand, if $y<x$ then $F(y) \leq F(x)$, and since $F$ is injective $F(y)<F(x)$. Therefore $G(F(x))=x$.

$(1) \implies (2)$. Since $F(X) \subseteq A$ and $G(F(x))=x$ for each $x \in X$, it follows that $F$ is injective. Now, we prove that $A \subseteq F(X)$. Indeed, let $r \in A$, then $F(G(r))=r$, so $r \in F(X)$.

$(1) \implies (3)$. Since $G(F(x))=x$ for each $x \in X$, it follows that $G$ is surjective. Since $F(G(r))=r$ for each $r \in A$, it follows that $G$ is injective.

$(1) \implies (4)$. Since $(1)$ implies $(2)$ and $(3)$, we have that $F$ and $G$ are both injective, so $(4)$ follows from Propositions \ref{pro:giny} and \ref{pro:finy}.

$(2) \implies (1)$. Let $r \in A$. Since $F(X)=A$ and $F$ is injective, there is only one $x \in X$ such that $F(x)=r$. It follows by definition of $G$ that $G(r)=x$ and hence $F(G(r))=F(x)=r$. By the Claim we have the rest of item $(1)$.

$(3) \implies (1)$. Let $r \in A$, then $F(G(r)) \geq r$ by Proposition \ref{pro:inverse}. Suppose that $F(G(r))>r$. It easily follows that $]r,F(G(r))[ \subseteq A$ and $G(]r,F(G(r))[)=G(r)$, but this is a contradiction, since $G$ is injective. We conclude that $F(G(r))=r$.

Now, let $x \in X$. Since $G$ is bijective, there exists $r \in A$ such that $x=G(r)$. It follows that $F(x)=F(G(r))=r$ and hence $F(x) \in A$. Therefore $F(X) \subseteq A$.

Finally, let $x \in X$, then $F(x) \in A$ and $G(F(x))=\inf \{y \in X:F(y) \geq F(x) \}$. Suppose that there exists $y<x$ such that $F(y) \geq F(x)$. By monotonicity of $F$ it follows that $F(y)=F(x)$. Since $G$ is biyective, there exists $r,s \in [0,1]$ such that $G(r)=y$ and $G(s)=x$. Note that $r<s$ by monotonicity of $G$. It follows that $r=F(G(r))=F(y)=F(x)=F(G(s))=s$, a contradiction. We conclude that $x=\inf \{y \in X:F(y) \geq F(x) \}=G(F(x))$.

$(4) \implies (1)$. By Corollary \ref{cor:fgrr} it follows that $F(G(r))=r$ for each $r \in A$. By Proposition \ref{pro:finy}, $F$ is injective and by the Claim it follows that $F(X) \subseteq A$ and $G(F(x))=x$ for each $x \in X$.

\end{proof}

\begin{pro}\label{pro:G^{-1}}
	Let $a, b \in X$ be such that $a<b$, then $G^{-1}(]a,b[)=]F(a), F_-(b)| \cap A$, where $|$ means $]$ or $[$ and $A$ is the subset of $[0,1]$ where $G$ is defined.
	\begin{proof}
		First of all, we show that $G^{-1}(]a,b[) \subseteq ]F(a), F_-(b)] \cap A$. For that purpose, let $r \in G^{-1}(]a,b[)$ and suppose that $r \notin ]F(a), F_-(b)]$, then it can happen:
		\begin{itemize}
			\item $r \leq F(a)$ what implies, by Corollary \ref{cor:inverse}, that $G(r) \leq a$ what is a contradiction with the fact that $r \in G^{-1}(]a,b[)$.
			\item $r>F_-(b)$ what gives us, by Proposition \ref{pro:inverse:F}.2, that $b \leq G(r)$ what implies that $r \notin G^{-1}(]a,b[)$ since $b \notin ]a,b[$, a contradiction.
			\end{itemize}
			
			Now we prove that $]F(a), F_-(b)[ \cap A \subseteq G^{-1}(]a,b[)$. For that purpose, let $r \in ]F(a), F_-(b)[$ where $G$ is defined, and suppose that $r \notin G^{-1}(]a,b[)$, then it can happen:
				\begin{itemize}
					\item $G(r) \leq a$ what implies, by Corollary \ref{cor:inverse}, that $r \leq F(a)$, a contradiction with the fact that $r>F(a)$.
					\item $G(r) \geq b$ what gives us, by Proposition \ref{pro:inverse:F}.4, that $r \geq F_-(b)$, a contradiction with the fact that $r<F_-(b)$.
				\end{itemize}
				
	\end{proof}
\end{pro}

According to the previous proposition, it is clear the next

\begin{cor}\label{cor:measurable}
	Suppose that $G$ is defined on $[0,1]$. Let $a, b \in X$ be such that $a<b$. Then $G^{-1}(]a,b[) \in \sigma([0,1])$, where $\sigma([0,1])$ denotes de Borel $\sigma$-algebra with respect to the euclidean topology.
	
	\begin{proof}
		Since $G^{-1}(]a,b[)=]F(a), F_-(b)|$, it is an open set or the intersection of an open and a closed set. Consequently, $G^{-1}(]a,b[) \in \sigma([0,1])$.
	\end{proof}
\end{cor}

\begin{pro}\label{pro:opens}
	Each open set in $\tau$ is the countable union of open intervals.
	\begin{proof}
		Let $G \subseteq X$ be an open set in $\tau$. If $G=\emptyset$, the result is clear since it can be written as $G=]a,a[$. Now suppose that $G$ is nonempty, then $G=\bigcup_{i \in I} G_i$, where $G_i$ is a convex component of $G$ for each $i \in I$ (see Proposition \ref{pro:convexcomp}). Now we prove that $G_i$ is open for each $i \in I$. Let $i \in I$ and $x \in G_i$. Since $G$ is an open set and $\{]a,b[: a, b \in X, a<b\}$ is an open basis of $X$ with respect to $\tau$ there exist $a, b \in X$ such that $x \in ]a,b[ \subseteq G$. Note that $G_i \cup ]a,b[$ is a convex set contained in $G$, what implies that $G_i \cup ]a,b[=G_i$ since $G_i$ is a convex component of $G$. Consequently, $]a,b[ \subseteq G_i$, what means that $G_i$ is an open set. Now, let $D$ be a countable dense subset of $X$, then we can choose $d_i \in D \cap G_i$ for each $i \in I$ what gives us the countability of $I$, since the familily $\{G_i: i \in I\}$ is pairwise disjoint.

Since $G_i$ is convex and open, by Corollary \ref{cor:intervals}, $G_i$ can be written as a countable union of open intervals. Thus, $G$ is the countable union of open intervals.
		\end{proof}
\end{pro}

Next result will be essential to show that $G$ is measurable with respect to the Borel $\sigma$-algebra.

\begin{teo}(\cite[Th. 1.7.2]{Bauer}) \label{Bauer}
	Let $(\Omega, \mathfrak{A})$ and $(\Omega', \mathfrak{A}')$ be measurable spaces; further let $\mathfrak{B}'$ be a generator of $\mathfrak{A}'$. A mapping $T: \Omega \rightarrow \Omega'$ is measurable if, and only if $T^{-1}(A') \in \mathfrak{A}$, for each $A' \in \mathfrak{B}'$.
\end{teo}

Since $G^{-1}(]a,b[) \in \sigma([0,1])$, for each $a, b \in X$ with $a<b$ and by taking into account Proposition \ref{pro:opens}, we conclude that

\begin{cor}\label{cor:Gmeasurable}
	Suppose that $G$ is defined on $[0,1]$. Then $G$ is measurable with respect to the Borel $\sigma$-algebras.
	
	\begin{proof}
		
		To show that $G$ is measurable we just have to use Corollary \ref{cor:measurable}, Theorem \ref{Bauer} and the fact that each open set in $\tau$ can be written as countable union of open intervals (see Proposition \ref{pro:opens}).
	\end{proof}
\end{cor}

\section{Generating samples}

\begin{lema}\label{le:alg}
	The family $\mathfrak{A}=\{\bigcup_{i=1}^n |a_i, b_i|: a_1 \leq b_1<a_2 \leq b_2<\ldots<a_n \leq b_n, a_1 \in X \cup \{-\infty\}, b_n \in X \cup \{\infty\}\}$ is an algebra and the $\sigma$-algebra generated by it is the Borel $\sigma$-algebra.
\end{lema}
\begin{proof}
	 Now we prove that $\mathfrak{A}$ is an algebra.
	\begin{enumerate}
		\item $A \cup B \in \mathfrak{A}$, for each $A, B \in \mathfrak{A}$. Indeed, this is true due to the fact that the union of two intervals consists on two disjoint intervals in case $A \cap B =\emptyset$ or it is a new interval otherwise. 
		\item $A \cap B \in \mathfrak{A}$, for each $A, B \in \mathfrak{A}$. Indeed, this is true due to the fact that the intersection of two intervals is $\emptyset$ or a new interval. Hence, $A \cap B$ is finite union of disjoint intervals, what means that $A \cap B \in \mathfrak{A}$.
		\item $\X \backslash A \in \mathfrak{A}$, for each $A \in\mathfrak{A}$. Indeed, this is true due to the fact that $\X \backslash A=]-\infty,a_{1}| \cup |b_1, a_2| \cup \ldots \cup |b_{n-1}, a_n| \cup |b_{n},\infty[  \in \mathfrak{A}$.
	
	\end{enumerate}
	
	Note that each element in $\mathfrak{A}$ belongs to $(X,\tau)$. Indeed, this is true due to the fact that, given $A \in \mathfrak{A}$, it consists of the finite union of open intervals, semi-open intervals (which are the intersection of an open and a closed set) or closed intervals (which are closed). Hence, $\mathfrak{S}$ is contained in the Borel $\sigma$-algebra of $(X,\tau)$, where $\mathfrak{S}=\sigma(\mathfrak{A})$.
	Finally, if $G$ is an open set in $(\X,\tau)$, by Proposition \ref{pro:opens}, it can be written as the countable union of open intervals. Thus $G$ can be written as the countable union of elements in $\mathfrak{A}$, what means that $G \in \mathfrak{S}$.  We conclude that $\mathfrak{S}$ is the Borel $\sigma$-algebra of $(X,\tau)$.
	
\end{proof}

Now, we want to prove the uniqueness of the measure with respect to its cdf.

First, we recall from \cite{Halmos} a theorem about the uniqueness of a measure. As a consequence of the next theorem we have that two measures that coincide in an algebra also coincide in its generated $\sigma$-algebra.

\begin{teo}(\cite[Chapter III, Th. A]{Halmos}) \label{teo:algebra}
	If $\mu$ is a $\sigma$-finite measure on a ring $R$, then there is a unique measure $\overline{\mu}$ on the $\sigma$-ring $S(R)$ such that, for $E$ in $R$, $\overline{\mu}(E)=\mu(E)$; the measure $\overline{\mu}$ is $\sigma$-finite.
\end{teo}

\begin{pro}
	Let $F_\mu$ and $F_\delta$ be the cdf's of the measures $\mu$ and $\delta$ satisfying $F_\mu=F_\delta$, then $\mu=\delta$ on the Borel $\sigma$-algebra of $(X, \tau)$.
	\begin{proof}
		Let $a, b \in X$ be such that $a \leq b$, then a cdf let us determine the measure of the set $|a,b|$. Indeed, we distinguish four cases depending on whether $a$ and $b$ belongs to $|a, b|$ or not.
		\begin{enumerate}
			\item $\mu(]a,b])=F_\mu(b)-F_\mu(a)=F_\delta(b)-F_\delta(a)=\delta(]a,b])$.
			\item $\mu([a,b])=F_{\mu}(b)-F_{\mu-}(a)=F_{\mu}(b)-\sup F_\mu(<a)=F_{\delta}(b)-\sup F_\delta(<a)=F_{\delta}(b)-F_{\delta-}(a)=\delta([a,b])$, where we have taken into account that $F_-(x)=\sup F(<x)$, for each $x \in X$ (see Proposition \ref{pro:sup}).
			\item $\mu(]a,b[)=F_{\mu-}(b)-F_{\mu}(a)=\sup F_{\mu}(<b)-F_\mu(a)=\sup F_{\delta}(<b)-F_\delta(a)=F_{\delta-}(b)-F_{\delta}(a)=\delta(]a,b[)$.
			\item $\mu([a,b[)=F_{\mu-}(b)-F_{\mu-}(a)=\sup F_{\mu}(<b)-\sup F_\mu(<a)=\sup F_{\delta}(<b)-\sup F_\delta(<a)=F_{\delta-}(b)-F_{\delta-}(a)=\delta([a,b[)$.
		\end{enumerate}
		
		Since $\mu(|a,b|)=\delta(|a,b|)$, for each $a, b \in X$ with $a \leq b$, it follows that $\mu(A)=\delta(A)$, for each $A \in \mathfrak{A}$, due to the $\sigma$-additivity of $\mu$ and $\delta$ as measures. Since $\mu=\delta$ on $\mathfrak{A}$, we conclude that $\mu=\delta$ on $\sigma(\mathfrak{A})$, that is, they coincide on the Borel $\sigma$-algebra of $(X, \tau)$ by the previous results.
	\end{proof}
\end{pro}

\begin{teo}(\cite[Th A. 81.]{Sch}) \label{teo:fun}
	A measurable function $f$ from one measure space $(S_1, \mathcal{A}_1, \mu_1)$ to a measurable space $(S_2, \mathcal{A}_2)$, $f:S_1 \rightarrow S_2$, induces a measure on the range $S_2$. For each, $A \in \mathcal{A}_2$, define $\mu_2(A)=\mu_1(f^{-1}(A))$. Integrals with respect to $\mu_2$ can be written as integrals with respect to $\mu_1$ in the following way if $g: S_2 \rightarrow \R$ is integrable, then,
	
	$$\displaystyle \int g(y) d\mu_2(y)=\displaystyle \int g(f(x)) d\mu_1(x)$$
\end{teo}

\begin{pro}
	Let $\mu$ be a probability measure and suppose that $G$ is defined on $[0,1]$. Then $\mu(A)=l(G^{-1}(A))$ for each $A \in \sigma([0,1])$, where $l$ is the Lebesgue measure and $\sigma([0,1])$ is the Borel $\sigma$-algebra of $[0,1]$.
	
	\begin{proof}
		By Proposition \ref{pro:G^{-1}}, we have that $G^{-1}(]a,b[)=\left]F(a), F_-(b)\right|$, for each $a, b \in X$ with $a<b$.
		
		Moreover, by Corollary \ref{cor:measure}, it holds that $\mu(]a,b[)=F_-(b)-F(a)$. It follows that $l(G^{-1}(]a,b[))=\mu(]a,b[)$ for each $a, b \in X$ with $a<b$.
		
		Now let $\mu_2$ be the measure defined by $\mu_2(A)=l(G^{-1}(A))$, for each $A \in \sigma([0,1])$. Indeed, $\mu_2$ is a measure by Theorem \ref{teo:fun} and Corollary \ref{cor:Gmeasurable}. Note that the fact that $\mu(]a,b[)=\mu_2(]a,b[)$, for each $a,b \in \X$ with $a<b$, implies that $\mu=\mu_2$ on the algebra $\mathfrak{A}$. Therefore $\mu$ and $\mu_2$ coincides in an algebra which generates $\sigma([0,1])$, so they are equal in $\sigma([0,1])$ (see for example Theorem \ref{teo:algebra}).
		
		Consequently, we can write $\mu(A)=l(G^{-1}(A))$ for each $A \in \sigma([0,1])$.
	\end{proof}
\end{pro}

Finally, by taking into account the previous results, we can generate samples with respect to the probability measure $\mu$ by following the classical procedure (see Section 1). In our case we will have to use $G$ to do it.

\begin{ob}Suppose that $G$ is defined on $[0,1]$. We can also calculate integrals with respect to $\mu$ by using Theorem \ref{teo:fun}, so for $g: X \rightarrow \R$,
	
	$$\displaystyle \int g(x) d\mu(x)=\displaystyle \int g(G(t)) dt$$
\end{ob}

\begin{ob}
	Suppose that $X$ is compact, then every subset of $X$ has both infimum and supremum (see Proposition \ref{pro:complete}) and hence $G$ is defined in each point of $[0,1]$. Therefore, in this case, we can generate samples with respect to a distribution based on a measure $\mu$.
\end{ob}

\begin{ob}
	Note that the classical theory for the distribution function is a particular case of the one we have developed for a separable LOTS.
	\end{ob}

\end{document}